\documentclass{amsart}
\DeclareFontEncoding{OT2}{}{}
\DeclareFontSubstitution{OT2}{cmr}{m}{n}

\DeclareFontFamily{OT2}{cmr}{\hyphenchar\font45 }
\DeclareFontShape{OT2}{cmr}{m}{n}{%
   <5><6><7><8><9>gen*wncyr%
   <10><10.95><12><14.4><17.28><20.74><24.88>wncyr10}{}

\DeclareMathAlphabet{\mathcyr}{OT2}{cmr}{m}{n}
\DeclareMathAlphabet{\mathcyb}{OT2}{cmr}{b}{n}
\SetMathAlphabet{\mathcyr}{bold}{OT2}{cmr}{b}{n}

\usepackage[dvipdfmx]{graphicx}
\newtheorem{thm}{Theorem}[section]
\newtheorem{lem}[thm]{Lemma}
\newtheorem{prop}[thm]{Proposition}

\theoremstyle{definition}
\newtheorem{defn}[thm]{Definition}

\theoremstyle{remark}
\newtheorem{rem}[thm]{Remark}
\newcommand{\A}{\mathcal{A}}
\newcommand{\F}{\mathcal{F}}
\newcommand{\FF}{\mathbb{F}}
\newcommand{\Q}{\mathbb{Q}}
\newcommand{\R}{\mathbb{R}}
\newcommand{\SSS}{\mathcal{S}}
\newcommand{\Z}{\mathbb{Z}}
\newcommand{\ZZ}{\mathfrak{Z}}
\newcommand{\ZZZ}{\mathcal{Z}}

\newcommand{\I}{\mathcal{I}}
\newcommand{\harb}{\mathbin{\overline{\ast}}}
\newcommand{\sha}{\mathbin{\overline{\mathcyr{sh}}}}

\DeclareMathOperator{\dep}{dep}
\DeclareMathOperator{\wt}{wt}

\begin{document}
\title[Ohno type relations for MZSVs]{Ohno type relations for classical and finite multiple zeta-star values}
\author{Minoru Hirose}
\address[Minoru Hirose]{Faculty of Mathematics, Kyushu University
 744, Motooka, Nishi-ku, Fukuoka, 819-0395, Japan}
\email{m-hirose@math.kyushu-u.ac.jp}
\author{Kohtaro Imatomi}
\address[Kohtaro Imatomi]{}
\email{k-imatomi@math.kyushu-u.ac.jp}
\author{Hideki Murahara}
\address[Hideki Murahara]{Nakamura Gakuen University Graduate School,
 5-7-1, Befu, Jonan-ku, Fukuoka, 814-0198, Japan} 
\email{hmurahara@nakamura-u.ac.jp}
\author{Shingo Saito}
\address[Shingo Saito]{Faculty of Arts and Science, Kyushu University,
 744, Motooka, Nishi-ku, Fukuoka, 819-0395, Japan}
\email{ssaito@artsci.kyushu-u.ac.jp}

\begin{abstract}
Ohno's relation is a generalization of both the sum formula and the duality formula for multiple zeta values.
Oyama gave a similar relation for finite multiple zeta values, defined by Kaneko and Zagier.
In this paper, we prove relations of similar nature for both multiple zeta-star values and finite multiple zeta-star values.
Our proof for multiple zeta-star values uses the linear part of Kawashima's relation.
\end{abstract}
\maketitle

\section{Introduction}
\subsection{Classical multiple zeta(-star) values}
 For $k_1,\dots,k_r\in \Z_{\ge1}$ with $k_r \ge 2$, the multiple zeta values (MZVs) and the multiple zeta-star values (MZSVs) are defined by 
\begin{align*}
 \zeta(k_1,\dots, k_r)=\sum_{1\le n_1<\cdots <n_r} \frac {1}{n_1^{k_1}\cdots n_r^{k_r}}, \\
 \zeta^\star (k_1,\dots, k_r)=\sum_{1\le n_1\le \cdots \le n_r} \frac {1}{n_1^{k_1}\cdots n_r^{k_r}}. 
\end{align*} 
The MZ(S)Vs are known to satisfy the sum formula (Proposition \ref{sum}) and the duality formula (Proposition \ref{duality}). 
We write $(\{1\}^m)=( \underbrace{1,\ldots ,1}_{m})$. 
\begin{prop}[Sum formula; Granville \cite{granville_1997}, Zagier] \label{sum}
 For $k,r\in\Z_{\ge1}$ with $k\ge r+1$, we have
 \begin{align*}
 \sum_{\substack{k_1+\dots+k_r=k\\k_r\ge2,k_i\ge1(1\le i \le r)}} \zeta (k_1,\dots,k_r)&= \zeta(k), \\
 \sum_{\substack{k_1+\dots+k_r=k\\k_r\ge2,k_i\ge1(1\le i \le r)}} \zeta^\star (k_1,\dots,k_r)&= \binom{k-1}{r-1} \zeta(k).
 \end{align*}
\end{prop}
An index is a (possibly empty) sequence of positive integers.
An index is said to be admissible if either it is empty or its last component is greater than $1$.
\begin{defn}
If we write an admissible index $\boldsymbol{k}$ as
 \[
  \boldsymbol{k}=(\{1\}^{a_1-1},b_1+1,\dots,\{1\}^{a_s-1},b_s+1) \quad (a_p, b_q\ge1),
 \]
we define the dual index of $\boldsymbol{k}$ as 
 \[
  \boldsymbol{k}^\dagger =(\{1\}^{b_s-1},a_s+1,\dots,\{1\}^{b_1-1},a_1+1).
 \]
\end{defn}
\begin{prop}[Duality formula] \label{duality}
 For an admissible index $\boldsymbol{k}$,  we have
 \[ \zeta(\boldsymbol{k}^\dagger)=\zeta(\boldsymbol{k}). \]
\end{prop}
Ohno's relation (Theorem \ref{ohno}) is a generalization of Propositions \ref{sum} and \ref{duality}. 
For a sequence $\boldsymbol{k}=(k_1,\ldots,k_r)$, let $\wt(\boldsymbol{k})=k_1+\cdots+k_r$, called its weight, and $\dep(\boldsymbol{k})=r$, called its depth. 
When two sequences $\boldsymbol{k}$ and $\boldsymbol{l}$ have the same depth, we denote by $\boldsymbol{k} \oplus \boldsymbol{l}$ the componentwise sum of $\boldsymbol{k}$ and $\boldsymbol{l}$.

In the following theorem and afterwards, we will always assume that $\boldsymbol{e}$ runs over sequences of nonnegative integers.
\begin{thm}[Ohno's relation; Ohno \cite{ohno_99}] \label{ohno}
 For an admissible index $\boldsymbol{k}$ and $m\in\Z_{\ge 0}$, we have
 \begin{align*}
  \sum_{\substack{ \wt (\boldsymbol{e})=m \\ \dep (\boldsymbol{e})=\dep (\boldsymbol{k}) }}
  \zeta (\boldsymbol{k}\oplus\boldsymbol{e}) 
  =\sum_{\substack{ \wt (\boldsymbol{e})=m \\ \dep (\boldsymbol{e})=\dep(\boldsymbol{k}^\dagger)}}
  \zeta (\boldsymbol{k}^\dagger \oplus\boldsymbol{e}).
 \end{align*}
\end{thm}

Our first main theorem is an analogue of Ohno's relation for MZSVs.
\begin{thm}[Main theorem $1$] \label{ohnoS}
 For an admissible index $\boldsymbol{k}$ and $m\in\Z_{\ge 0}$, we have
 \[
  \sum_{\substack{ \wt(\boldsymbol{e})=m \\ \dep (\boldsymbol{e})=\dep(\boldsymbol{k}) }}
  c_1(\boldsymbol{k},\boldsymbol{e})\zeta^\star (\boldsymbol{k}\oplus\boldsymbol{e}) 
  =\sum_{\substack{ \wt(\boldsymbol{e})=m \\ \dep(\boldsymbol{e})=\dep(\boldsymbol{k}^\dagger ) }} \zeta^\star ((\boldsymbol{k}^\dagger\oplus\boldsymbol{e} )^\dagger),
 \]
 where
 \begin{align*}
  c_1((k_1\ldots,k_r),(e_1,\ldots,e_r))=\prod_{i=1}^{r}
  \binom{k_i+e_i+\delta_{i,1}-2 }{e_i},\,\, 
  \binom{n-1}{n}=
   \begin{cases}
    1 & \text{if $n=0$}; \\
    0 & \text{otherwise}.
   \end{cases}    
 \end{align*}
\end{thm}

\begin{rem}
 The second formula of Proposition \ref{sum} easily follows from Theorem~\ref{ohnoS}.
 For details, see Section 3.
\end{rem}

\subsection{Finite multiple zeta values}
Kaneko and Zagier \cite{kaneko_zagier_2017} introduced the finite multiple zeta(-star) values (FMZ(S)Vs) and
the symmetric multiple zeta(-star) values (SMZ(S)Vs). 
Set $\A=\prod_p\FF_p/\bigoplus_p\FF_p$, where $p$ runs over all primes.
For $k_1,\dots,k_r\in\Z_{\ge1}$, we define
\begin{align*}
 \zeta_{\A}(k_1,\dots,k_r)&=\Biggl(\sum_{1\le m_1<\dots<m_r<p}\frac{1}{m_1^{k_1}\dotsm m_r^{k_r}}\bmod p\Biggr)_p\in\A,\\
 \zeta_{\A}^{\star}(k_1,\dots,k_r)&=\Biggl(\sum_{1\le m_1\le\dots\le m_r<p}\frac{1}{m_1^{k_1}\dotsm m_r^{k_r}}\bmod p\Biggr)_p\in\A.
\end{align*}
Let $\ZZZ$ denote the $\Q$-linear subspace of $\R$ spanned by the multiple zeta values.
For $k_1,\dots,k_r\in\Z_{\ge1}$, we define
\begin{align*}
 \zeta_{\SSS}(k_1,\dots,k_r)&=\sum_{i=0}^{r}(-1)^{k_{i+1}+\dots+k_r}\zeta(k_1,\dots,k_i)
 \zeta(k_r,\dots,k_{i+1})\bmod\zeta(2)\in\ZZZ/\zeta(2)\ZZZ,\\
 \zeta_{\SSS}^{\star}(k_1,\dots,k_r)&=\sum_{j=0}^{r}(-1)^{k_{i+1}+\dots+k_r}\zeta^{\star}(k_1,\dots,k_i)\zeta^{\star}(k_r,\dots,k_{i+1})\bmod\zeta(2)\in\ZZZ/\zeta(2)\ZZZ,
\end{align*}
where we set $\zeta(\emptyset)=\zeta^{\star}(\emptyset)=1$.
The MZ(S)Vs that appear in the definition of the SMZ(S)Vs
are the regularized values if the last component is $1$;
although there are two ways of regularization, called the harmonic regularization and the shuffle regularization (see \cite{ihara_kaneko_zagier_2006}),
it is known that the SMZVs remain unchanged as elements of $\ZZZ/\zeta(2)\ZZZ$
no matter which regularization we use (see \cite{kaneko_zagier_2017}).

Kaneko and Zagier \cite{kaneko_zagier_2017} made a conjecture that the FMZVs and the SMZVs are isomorphic;
more precisely, if we let $\ZZZ_{\A}$ denote the $\Q$-linear subspace of $\A$ spanned by the FMZVs,
then $\ZZZ_{\A}$ and $\ZZZ/\zeta(2)\ZZZ$ are isomorphic as $\Q$-algebras via the correspondence
$\zeta_{\A}(k_1,\dots,k_r)\leftrightarrow\zeta_{\SSS}(k_1,\dots,k_r)$.
It means that $\zeta_{\A}(k_1,\dots,k_r)$ and $\zeta_{\SSS}(k_1,\dots,k_r)$ satisfy the same relations.
In what follows, we use the letter $\F$ when it can be replaced with either $\A$ or $\SSS$;
for example, by $\zeta_{\F}(1)=0$ we mean that both $\zeta_{\A}(1)=0$ and $\zeta_{\SSS}(1)=0$ are true.
We write
\[
 \ZZ_{\F}(k)=
 \begin{cases}
  (B_{p-k}/k\bmod p)_p&\text{if $\F=\A$;}\\
  \zeta(k)\bmod\zeta(2)&\text{if $\F=\SSS$}
 \end{cases}
\]
for $k\in\Z_{\ge2}$, where $B_n$ denotes the $n$-th Bernoulli number.
Note that it can be verified rather easily that $\zeta_{\F}(1,k-1)=\ZZ_{\F}(k)$ for $k\in\Z_{\ge2}$,
so that $(B_{p-k}/k\bmod p)_p$ corresponds to $\zeta(k)\bmod\zeta(2)$ via the above-mentioned isomorphism $\ZZZ_{\A}\cong\ZZZ/\zeta(2)\ZZZ$.

The FMZ(S)Vs and the SMZ(S)Vs are known to satisfy the sum formula (Proposition \ref{sumFS}) and the duality formula (Proposition \ref{dualFS}).
\begin{prop}[Sum formula; Saito-Wakabayashi \cite{saito_wakabayashi_2015}, Murahara \cite{murahara_2015}] \label{sumFS}
 For $k,r,i\in\Z$ with $1\le i\le r\le k-1$, we have
 \begin{align*}
  \sum_{\substack{k_1+\dots+k_r=k\\k_i\ge2,k_j\ge1(1\le j\le r)}}\zeta_{\F} (k_1,\dots,k_r)
  =(-1)^{i}\biggl(\binom{k-1}{i-1}+(-1)^r\binom{k-1}{r-i}\biggr)\ZZ_{\F}(k), \\
  \sum_{\substack{k_1+\dots+k_r=k\\k_i\ge2,k_j\ge1(1\le j\le r)}}\zeta_{\F}^{\star}(k_1,\dots,k_r)
  =(-1)^{i}\biggl((-1)^r\binom{k-1}{i-1}+\binom{k-1}{r-i}\biggr)\ZZ_{\F}(k).
 \end{align*}
\end{prop}
\begin{defn}
If we write a nonempty index $\boldsymbol{k}$ as
 \[
  \boldsymbol{k}=(a_0,\{1\}^{b_1-1},a_1+1,\{1\}^{b_2-1},a_2+1,\ldots,\{1\}^{b_{s-1}-1},a_{s-1}+1,\{1\}^{b_s-1},a_s) \quad (a_p, b_q\ge1),
 \]
we define Hoffman's dual index of $\boldsymbol{k}$ as 
 \begin{align*}
  \boldsymbol{k}^\vee =(\{1\}^{a_0-1},b_1+1,\{1\}^{a_{1}-1},\ldots,b_s+1,\{1\}^{a_s-1}).
 \end{align*}
\end{defn}
\begin{prop}[Duality formula; Hoffman \cite{hoffman_2015}, Jarossay \cite{jarossay_2014}] \label{dualFS}
 For a nonempty index $\boldsymbol{k}$, we have
 \[ \zeta_{\mathcal{F}}^\star (\boldsymbol{k})=-\zeta_{\mathcal{F}}^\star (\boldsymbol{k}^{\vee}). \]
\end{prop}

The following Ohno type relation for FMZVs and SMZVs was conjectured by Kaneko \cite{kaneko_2015} and established by Oyama \cite{oyama_2015}.
\begin{thm}[Oyama \cite{oyama_2015}] \label{ohnoF}
 For a nonempty index $\boldsymbol{k}$ and $m\in\Z_{\ge 0}$, we have
 \[ 
  \sum_{\substack{ \wt(\boldsymbol{e})=m \\ \dep(\boldsymbol{e})=\dep(\boldsymbol{k}) }}
  \zeta_{\mathcal{F}} (\boldsymbol{k}\oplus\boldsymbol{e})
  =\sum_{\substack{ \wt(\boldsymbol{e})=m \\ \dep(\boldsymbol{e})=\dep(\boldsymbol{k}^\vee ) }}
  \zeta_{\mathcal{F}} ((\boldsymbol{k}^\vee \oplus\boldsymbol{e})^\vee ).
 \]
\end{thm}
\begin{rem}
 Although Theorem \ref{ohno} with $m=0$ gives Proposition~\ref{duality},
 Theorem~\ref{ohnoF} with $m=0$ only gives a trivial relation.
\end{rem}

Our second main theorem is an analogue of Ohno's relation for FMZSVs and SMZSVs.
It is a generalization of Propositions \ref{sumFS} and \ref{dualFS}.
\begin{thm}[Main theorem $2$] \label{ohnoFS}
 For a nonempty index $\boldsymbol{k}$ and $m\in\Z_{\ge 0}$, we have 
 \begin{align*}
  \sum_{\substack{ \wt(\boldsymbol{e})=m \\ \dep(\boldsymbol{e})=\dep(\boldsymbol{k}) }}
  c_2(\boldsymbol{k},\boldsymbol{e})\zeta_\mathcal{F}^\star(\boldsymbol{k}\oplus\boldsymbol{e}) 
  =-\sum_{\substack{ \wt(\boldsymbol{e})=m \\ \dep(\boldsymbol{e})=\dep(\boldsymbol{k}^\vee ) }} \zeta_\mathcal{F}^\star(\boldsymbol{k}^\vee\oplus\boldsymbol{e} ),
 \end{align*}
 where
 \begin{align*}
  c_2((k_1\ldots,k_r),(e_1,\ldots,e_r))=\prod_{i=1}^{r}
  \binom{k_i+e_i+\delta_{i,1}+\delta_{i,r}-2 }{e_i},  \,\,
  \binom{n-1}{n}=
  \begin{cases}
    1 & \text{if $n=0$}; \\
    0 & \text{otherwise}.
  \end{cases} 
 \end{align*}
\end{thm}

\section{Proof of the main theorems}
\subsection{Shuffle and harmonic products}
We define two operators $R$, defined for all indices, and $P$, defined for all nonempty indices, by
\begin{align*}
 R(k_1,\ldots,k_r)&=(k_r,\ldots,k_1), \\
 P(k_1,\ldots,k_r)&=(k_1,\ldots,k_{r-1},k_r+1). 
\end{align*}

We denote by $\I$ the $\Q$-linear space spanned by the indices.
Everything that is defined for indices, such as $R$ and $\zeta$, will be extended $\Q$-linearly.
We define $\Q$-bilinear products $\sha$ and $\harb$ on $\I$ inductively by setting
\begin{gather*}
 \boldsymbol{k} \sha \emptyset = \emptyset \sha \boldsymbol{k}= \boldsymbol{k} \harb \emptyset = \emptyset \harb \boldsymbol{k} =\boldsymbol{k}, \\ 
 (k _1, \boldsymbol{k}) \sha (l_1,\boldsymbol{l}) = (k_1, \boldsymbol{k} \sha (l_1,\boldsymbol{l})) + (l_1, (k_1,\boldsymbol{k}) \sha \boldsymbol{l}), \\
 (k _1, \boldsymbol{k}) \harb (l_1,\boldsymbol{l}) = (k_1, \boldsymbol{k} \harb (l_1,\boldsymbol{l})) + (l_1, (k_1,\boldsymbol{k}) \harb \boldsymbol{l})- (k_1+l_1,\boldsymbol{k} \harb \boldsymbol{l})
\end{gather*} 
for all indices $\boldsymbol{k}$, $\boldsymbol{l}$ and all positive integers $k_1$, $l_1$
(for details on $\harb$, see Muneta \cite{muneta_2009}). 
\begin{lem} \label{lem1}
 For an index $\boldsymbol{k}$ and $m\in\Z_{\ge 0}$, we have
  \[ 
   \boldsymbol{k} \, \sha  (\{1\}^m )
   =\sum_{i=0}^{m} \sum_{\substack{ \wt(\boldsymbol{e})=m-i \\ \dep(\boldsymbol{e})=\dep(\boldsymbol{k}) }}
   (\boldsymbol{k}\oplus\boldsymbol{e}) \harb ( \{1\}^i).
 \]
\end{lem}
\begin{proof}
 Write $\boldsymbol{k}=(k_1,\ldots,k_r)$ and $\boldsymbol{e}=(e_1,\ldots,e_r)$.
 Then
 \begin{align*}
  \sum_{i=0}^{m} \sum_{\substack{ \wt(\boldsymbol{e})=m-i \\ \dep(\boldsymbol{e})=\dep(\boldsymbol{k}) }}
   (\boldsymbol{k}\oplus\boldsymbol{e}) \harb ( \{1\}^i)
  &=\sum_{i=0}^m \sum_{\substack{ e_1+\dots+e_r=m-i \\ e_l\ge0 (1\le l\le r) }} \sum_{j=0}^i (-1)^{i-j} \\
  &\qquad \sum_{\substack{ f_1+\dots+f_r=i-j \\ 0\le f_l\le1 (1\le l\le r) }} 
  (k_1+e_1+f_1,\ldots,k_r+e_r+f_r)\sha (\{1\}^j) \\
  &=\sum_{j=0}^m \sum_{\substack{ e_1+\dots+e_r+f_1+\dots+f_r=m-j \\ e_l\ge0, 0\le f_l\le1 (1\le l\le r) }} (-1)^{f_1+\dots +f_r} \\
  &\qquad (k_1+e_1+f_1,\ldots,k_r+e_r+f_r)\sha (\{1\}^j) \\
  &=(k_1,\ldots,k_r)\sha (\{1\}^m). \qedhere
 \end{align*} 
\end{proof}

\begin{lem}\label{lem:dual_c2}
 For a nonempty index $\boldsymbol{k}$ and $m\in\Z_{\ge0}$, we have
 \[
  (\boldsymbol{k}^\vee\sha (\{1\}^m))^\vee
  =\sum_{\substack{ \wt(\boldsymbol{e})=m \\ \dep(\boldsymbol{e})=\dep(\boldsymbol{k}) }}
  c_2(\boldsymbol{k},\boldsymbol{e})(\boldsymbol{k}\oplus\boldsymbol{e}) .
 \]
\end{lem}

\begin{proof}
 Write
 \[
  \boldsymbol{k}=(a_0,\{1\}^{b_1-1},a_1+1,\{1\}^{b_2-1},a_2+1,\ldots,\{1\}^{b_{s-1}-1},a_{s-1}+1,\{1\}^{b_s-1},a_s) \quad (a_p, b_q\ge1).
 \]
 Then
 \begin{align*}
  &(\boldsymbol{k}^\vee\sha (\{1\}^m))^\vee\\
  &=((\{1\}^{a_0-1},b_1+1,\{1\}^{a_{1}-1},\ldots,b_s+1,\{1\}^{a_s-1})\sha (\{1\}^m))^\vee  \\
  &=\sum_{\substack{ e_0+\cdots+e_s=m \\ e_i\ge0(1\le i\le s) }}
   (\{1\}^{a_0+e_0-1},b_1+1,\{1\}^{a_{1}+e_1-1},\ldots,b_s+1,\{1\}^{a_s+e_s-1})^\vee\prod_{i=0}^s \binom{a_i+e_i-1}{e_i}\\
  &=\sum_{\substack{ e_0+\cdots+e_s=m \\ e_i\ge0(1\le i\le s) }}
   (a_0+e_0,\{1\}^{b_1-1},a_1+e_1+1,\{1\}^{b_2-1},a_2+e_2+1,\ldots,\\
  &\hphantom{=\sum_{\substack{ e_0+\cdots+e_s=m \\ e_i\ge0(1\le i\le s) }}(}
   \{1\}^{b_{s-1}-1},a_{s-1}+e_{s-1}+1,\{1\}^{b_s-1},a_s+e_s)\prod_{i=0}^s \binom{a_i+e_i-1}{e_i}\\
  &=\sum_{\substack{ \wt(\boldsymbol{e})=m \\ \dep(\boldsymbol{e})=\dep(\boldsymbol{k}) }}
  c_2(\boldsymbol{k},\boldsymbol{e})(\boldsymbol{k}\oplus\boldsymbol{e}).
 \end{align*}
 Here, the last equality follows from the following observation:
 if $\boldsymbol{e}$ is a sequence of nonnegative integers with $\dep(\boldsymbol{e})=\dep(\boldsymbol{k})$,
 and we write
 \[
  \boldsymbol{e}=(e_0,e_{1,1},\dots,e_{1,b_1-1},e_1,e_{2,1},\dots,e_{2,b_2-1},e_2,\dots,e_{s,1},\dots,e_{s,b_s-1},e_s),
 \]
 then $c_2(\boldsymbol{k},\boldsymbol{e})\ne0$ only if $e_{i,j}=0$ for all $i=1,\dots,s$ and $j=1,\dots,b_i-1$, in which case
 \[
  c_2(\boldsymbol{k},\boldsymbol{e})=\prod_{i=0}^s \binom{a_i+e_i-1}{e_i}.\qedhere
 \]
\end{proof}

\subsection{Proof of Theorem \ref{ohnoS}}
The main ingredient of the proof of Theorem \ref{ohnoS} is the linear part of Kawashima's relation.
For a nonempty index $\boldsymbol{k}$, we define $\zeta^{\star,+}(\boldsymbol{k})=\zeta^{\star}(P(\boldsymbol{k}))$.
\begin{thm}[Linear part of Kawashima's relation; Kawashima \cite{kawashima_2009}] \label{kawashima_lin}
 For nonempty indices $\boldsymbol{k}$ and $\boldsymbol{l}$, we have 
 \begin{align*}
  \zeta^{\star,+} ((\boldsymbol{k}\harb\boldsymbol{l})^\vee)=0.
 \end{align*}
\end{thm}
\begin{lem} \label{ohnoS2}
 For a nonempty index $\boldsymbol{k}$ and $m\in\Z_{\ge 0}$, we have
 \begin{align*} \label{eq1} 
  \zeta^{\star,+} ((\boldsymbol{k} \sha \{1\}^m )^\vee)
  =\sum_{\substack{ \wt(\boldsymbol{e})=m \\ \dep(\boldsymbol{e})=\dep(\boldsymbol{k}) }}
  \zeta^{\star,+} ((\boldsymbol{k}\oplus\boldsymbol{e})^\vee ).
 \end{align*}
\end{lem}
\begin{proof}
 By Lemma \ref{lem1} and Theorem \ref{kawashima_lin}, 
 \begin{align*}
  \zeta^{\star,+} ((\boldsymbol{k} \, \sha  (\{1\}^m ))^\vee)
  &=\sum_{i=0}^{m} \sum_{\substack{ \wt(\boldsymbol{e})=m-i \\ \dep(\boldsymbol{e})=\dep(\boldsymbol{k}) }}
  \zeta^{\star,+} (((\boldsymbol{k}\oplus\boldsymbol{e}) 
  \, \harb   
  ( \{1\}^i) )^\vee) \\
  &=\sum_{\substack{ \wt(\boldsymbol{e})=m \\ \dep(\boldsymbol{e})=\dep(\boldsymbol{k}) }}
  \zeta^{\star,+} ((\boldsymbol{k}\oplus\boldsymbol{e})^\vee). \qedhere
 \end{align*} 
\end{proof}
\begin{proof}[Proof of Theorem \ref{ohnoS}]
 Since the theorem is trivial when $\boldsymbol{k}=\emptyset$, we may assume that $\boldsymbol{k}\neq\emptyset$, in which case 
 we can write $\boldsymbol{k}=P(\boldsymbol{l})$ for some nonempty index $\boldsymbol{l}$. 
 By Lemmas \ref{lem:dual_c2} and \ref{ohnoS2}, we have
 \begin{align*}
  \sum_{\substack{ \wt(\boldsymbol{e})=m \\ \dep(\boldsymbol{e})=\dep(\boldsymbol{l}^\vee ) }}
  \zeta^{\star,+} ((\boldsymbol{l}^\vee\oplus\boldsymbol{e})^\vee) 
  &=\zeta^{\star,+} ((\boldsymbol{l}^\vee\sha (\{1\}^m))^\vee ) \\
  &=\sum_{\substack{ \wt(\boldsymbol{e})=m \\ \dep(\boldsymbol{e})=\dep(\boldsymbol{l}) }}
  c_2(\boldsymbol{l},\boldsymbol{e})\zeta^{\star,+}(\boldsymbol{l}\oplus\boldsymbol{e})\\
  &=\sum_{\substack{ \wt(\boldsymbol{e})=m \\ \dep(\boldsymbol{e})=\dep(\boldsymbol{k}) }}
  c_1(\boldsymbol{k},\boldsymbol{e})\zeta^{\star} (\boldsymbol{k}\oplus\boldsymbol{e}).
 \end{align*}
 Since $\boldsymbol{k}^\dagger=PR(\boldsymbol{l}^\vee)$, we have 
 \[
  \sum_{\substack{ \wt(\boldsymbol{e})=m \\  \dep(\boldsymbol{e})=\dep(\boldsymbol{l}^\vee ) }}
  \zeta^{\star,+} ((\boldsymbol{l}^\vee\oplus\boldsymbol{e})^\vee)
  =\sum_{\substack{ \wt(\boldsymbol{e})=m \\ \dep(\boldsymbol{e})=\dep(\boldsymbol{k}^\dagger ) }}
  \zeta^{\star} ( (\boldsymbol{k}^\dagger\oplus\boldsymbol{e})^\dagger ).
 \]
 Thus, we find the desired result. 
\end{proof}

\subsection{Proof of Theorem \ref{ohnoFS}}
\begin{lem} \label{ohnoFS2}
 For an index $\boldsymbol{k}$ and $m\in\Z_{\ge 0}$, we have
 \[ 
  \zeta_\mathcal{F}^\star (\boldsymbol{k} \, \sha (\{1\}^m) )
  =\sum_{\substack{ \wt(\boldsymbol{e})=m \\ \dep(\boldsymbol{e})=\dep(\boldsymbol{k}) }}
  \zeta_\mathcal{F}^\star (\boldsymbol{k}\oplus\boldsymbol{e}). 
 \]
\end{lem}
\begin{proof}
 We note that $\zeta_\F^\star(\{1\}^i)=0$ for all $i\in\mathbb{Z}_{\ge1}$,
 and $\zeta_\F^\star(\boldsymbol{k}\harb\boldsymbol{l})=\zeta_\F^\star(\boldsymbol{k})\zeta_\F^\star(\boldsymbol{l})$
 for all indices $\boldsymbol{k}$ and $\boldsymbol{l}$. 
 By Lemma \ref{lem1}, we have
 \begin{align*}
  \zeta_\mathcal{F}^\star( \boldsymbol{k} \, \sha  (\{1\}^m ) )
  &=\sum_{i=0}^{m} \sum_{\substack{ \wt(\boldsymbol{e})=m-i \\ \dep(\boldsymbol{e})=\dep(\boldsymbol{k}) }}
  \zeta_\mathcal{F}^\star( (\boldsymbol{k}\oplus\boldsymbol{e}) 
  \, \harb   
  ( \{1\}^i) ) \\
  &=\sum_{\substack{ \wt(\boldsymbol{e})=m \\ \dep(\boldsymbol{e})=\dep(\boldsymbol{k}) }} \zeta_\mathcal{F}^\star(\boldsymbol{k}\oplus\boldsymbol{e}). \qedhere
 \end{align*} 
\end{proof}

\begin{proof}[Proof of Theorem \ref{ohnoFS}.]
 By Proposition \ref{dualFS} and Lemmas \ref{lem:dual_c2} and \ref{ohnoFS2}, we have 
 \begin{align*}
  -\sum_{\substack{ \wt(\boldsymbol{e})=m \\ \dep(\boldsymbol{e})=\dep(\boldsymbol{k}^\vee ) }}
  \zeta_\mathcal{F}^\star( \boldsymbol{k}^\vee\oplus\boldsymbol{e} ) 
  &=\zeta_\mathcal{F}^\star\left( (\boldsymbol{k}^\vee\sha (\{1\}^m))^\vee \right) \\
  &=\sum_{\substack{ \wt(\boldsymbol{e})=m \\ \dep(\boldsymbol{e})=\dep(\boldsymbol{k}) }}
  c_2(\boldsymbol{k},\boldsymbol{e})\zeta_\mathcal{F}^\star(\boldsymbol{k}\oplus\boldsymbol{e}). \qedhere
 \end{align*}
\end{proof}

\section{Applications}
In this section, we prove the second formula of Proposition \ref{sum} (resp.\ Proposition \ref{sumFS}) by using Theorem \ref{ohnoS} (resp.\ Theorem \ref{ohnoFS}).
\begin{prop}
 Theorem \ref{ohnoS} implies Proposition \ref{sum}.
\end{prop}
\begin{proof}
 By substituting $\boldsymbol{k}=(k)$ in Theorem \ref{ohnoS}, we have 
 \begin{align*}
  \text{L.H.S. of Theorem \ref{ohnoS}}  
  &=\binom{k+m-1}{m} \zeta^\star (k+m), \\
  \text{R.H.S. of Theorem \ref{ohnoS}} 
  &=\sum_{\substack{ \wt(\boldsymbol{e})=m \\ \dep(\boldsymbol{e})=k-1 }}
  \zeta^\star ( ((\{1\}^{k-2},2)\oplus\boldsymbol{e} )^\dagger ) \\
  &=\sum_{\substack{ \wt(\boldsymbol{e})=k-2 \\ \dep(\boldsymbol{e})=m+1 }}
  \zeta^\star ( (\{1\}^m,2)\oplus\boldsymbol{e} ). 
 \end{align*}
 This finishes the proof.
\end{proof}   

\begin{prop} \label{ohnoFS_to_sumFS}
 Theorem \ref{ohnoFS} implies Proposition \ref{sumFS}.
\end{prop}
To prove Proposition \ref{ohnoFS_to_sumFS}, we need the following lemma.
\begin{lem} \label{lem2}
 For $m,n,i\in\Z_{\ge1}$ with $i\le n$, we have 
\begin{align*}
 \sum_{a=0}^m (-1)^a \binom{m+n}{a+i} \binom{a+i-1}{a} \binom{m+n-a-i}{m-a}
 &=\binom{m+n}{m+i}, \\
 \sum_{a=0}^m (-1)^a \binom{m+n}{a+i-1} \binom{a+i-1}{a} \binom{m+n-a-i}{m-a}
 &=(-1)^m \binom{m+n}{i-1}. 
\end{align*}
\end{lem}
\begin{proof}
 Since  
 \begin{align*}
  \binom{m+n}{a+i} \binom{a+i-1}{a} \binom{m+n-a-i}{m-a}
  &=\frac{(m+n)!}{(n-i)! (i-1)!} \cdot\frac{1}{(a+1)a!(m-a)!}, \\
  \binom{m+n}{a+i-1} \binom{a+i-1}{a} \binom{m+n-a-i}{m-a}
  &=\frac{(m+n)!}{(n-i)! (i-1)!} \cdot\frac{1}{(m+n-a-i+1)a!(m-a)!}, 
 \end{align*}
 we need to show 
 \begin{align*}
  \sum_{a=0}^m (-1)^a \frac{1}{a+i} \binom{m}{a}
  &=\frac{m!(i-1)!}{(i+m)!}, \\
  \sum_{a=0}^m (-1)^a \frac{1}{m+n-a-i+1} \binom{m}{a}
  &=(-1)^m \frac{m!(n-i)!}{(m+n-i+1)!}.
 \end{align*}
 The first equation follows from
 \begin{align*}
  \sum_{a=0}^m (-1)^a \frac{1}{a+i} \binom{m}{a}
  &=\int_0^1 \sum_{a=0}^m (-1)^a \binom{m}{a} x^{a+i-1} dx \\
  &=\int_0^1 (1-x)^m x^{i-1} dx \\
  &=\frac{m!(i-1)!}{(i+m)!}, 
 \end{align*}
 and the second can be shown similarly. 
\end{proof}
\begin{proof}[Proof of Proposition \ref{ohnoFS_to_sumFS}]
 We note that $\zeta_\mathcal{F}^\star(k_1,k_2)=-(-1)^{k_1} \binom{k_1+k_2}{k_1} \mathfrak{Z}_\F(k_1+k_2)$.
 By substituting $\boldsymbol{k}=(i,l-i)$ in Theorem \ref{ohnoFS}, we have 
 \begin{align*}
  &\text{L.H.S. of Theorem \ref{ohnoFS}} \\
  &=\sum_{\substack{ e_1+e_2=m \\ e_1,e_2\ge0 }} \zeta_\F^\star(i+e_1,l-i+e_2)\binom{i+e_1-1}{e_1}\binom{l-i+e_2-1}{e_2} \\
  &=-\sum_{\substack{ e_1+e_2=m \\ e_1,e_2\ge0 }}
  (-1)^{i+e_1} \binom{l+m}{i+e_1}\binom{i+e_1-1}{e_1}\binom{l-i+e_2-1}{e_2} \mathfrak{Z}_\F(l+m) \\
  &=-(-1)^i \sum_{a=0}^m
  (-1)^{a} \left( \binom{l+m-1}{a+i} + \binom{l+m-1}{a+i-1} \right) \binom{a+i-1}{a} \\
   &\qquad\qquad\qquad\qquad\qquad\qquad\qquad\qquad\qquad\qquad \binom{l+m-a-i-1}{m-a} \mathfrak{Z}_\F(l+m) \\
  &=-(-1)^{i} \left( \binom{l+m-1}{m+i}+(-1)^m\binom{l+m-1}{i-1} \right) \mathfrak{Z}_\F(l+m) \qquad 
  \text{(by Lemma \ref{lem2})} 
 \end{align*}
 and 
 \begin{align*}
  \text{R.H.S. of Theorem \ref{ohnoFS}} 
  =-\sum_{\substack{k_1+\dots+k_{l-1}=l+m\\k_i\ge2,k_j\ge1(1\le j\le l-1)}}\zeta_{\F}^{\star}(k_1,\dots,k_{l-1}).
 \end{align*}
 Then, we get
 \begin{align*}
  &\sum_{\substack{k_1+\dots+k_{l-1}=l+m \\k_i\ge2,k_j\ge1(1\le j\le l-1)}} \zeta_{\F}^{\star}(k_1,\dots,k_{l-1}) \\
  &=(-1)^{i} \left( \binom{l+m-1}{m+i}+(-1)^m\binom{l+m-1}{i-1} \right) \mathfrak{Z}_\F(l+m)\\
  &=(-1)^{i}\biggl( (-1)^{l-1} \binom{l+m-1}{i-1} + \binom{l+m-1}{l-i-1} \biggr)\ZZ_{\F}(l+m).
 \end{align*}
 By putting $k=l+m$ and $r=l-1$, we have the desired result. 
\end{proof}   

\section*{Acknowledgements}
This work was supported by JSPS KAKENHI Grant Numbers JP18J00982, JP18K03243, and JP18K13392.

\end{document}